\documentclass[10pt]{amsart}

\usepackage{enumerate}
\usepackage{amsmath,amssymb}
\usepackage{color}

\newtheorem{theorem}{Theorem}[section]

\newtheorem*{q}{Question}
\newtheorem{corollary}[theorem]{Corollary}
\newtheorem{lemma}[theorem]{Lemma}

\newtheorem{proposition}[theorem]{Proposition}

\newtheorem{claim}[theorem]{Claim}

\theoremstyle{definition}

\def\M{{\mathbb{M}}}

\def\D{{\mathcal{D}}}
\def\P{{\mathcal{P}}}

\def\Z{{\mathcal{Z}}}

\def\otimesbar{{\mathbin{\overline{\otimes}}}}

\makeatletter
  
  \@addtoreset{equation}{section}
\makeatother

\usepackage{comment}

\newcommand{\vertiii}[1]{{\left\vert\kern-0.25ex\left\vert\kern-0.25ex\left\vert #1 
    \right\vert\kern-0.25ex\right\vert\kern-0.25ex\right\vert}}

\begin{document}

\title[Ring isomorphisms]{Ring isomorphisms of type II$_\infty$ locally measurable operator algebras}

\author[M. Mori]{Michiya Mori}

\address{Graduate School of Mathematical Sciences, The University of Tokyo, 3-8-1 Komaba Meguro-ku Tokyo, 153-8914, Japan; Interdisciplinary Theoretical and Mathematical Sciences Program (iTHEMS), RIKEN, 2-1 Hirosawa Wako Saitama, 351-0198, Japan.}
\email{mmori@ms.u-tokyo.ac.jp}

\thanks{The author was supported by RIKEN Special Postdoctoral Researcher Program and JSPS KAKENHI Grant Number 22K13934.}


\date{}

\begin{abstract}
We show that every ring isomorphism between the algebras of locally measurable operators for type II$_\infty$ von Neumann algebras is similar to a real $^*$-isomorphism. 
This together with previous results by the author and Ayupov--Kudaybergenov completely describes ring isomorphisms between the algebras of locally measurable operators as well as lattice isomorphisms between the projection lattices for a general pair of von Neumann algebras without finite type I direct summands.
\end{abstract}

\maketitle
\thispagestyle{empty}

\section{Introduction}
For a von Neumann algebra $M$, let $\P(M)$ denote the collection of all projections of $M$, i.e., $\P(M):=\{p\in M\mid p=p^*=p^2\}$. 
The importance of this collection in the investigation of the structure of a von Neumann algebra was first recognized by Murray and von Neumann in \cite{MN} in the 1930s. 
Recall that $\P(M)$ forms a lattice with respect to the partial order relation $p\leq q$ $\Leftrightarrow$ $pq=p$. 
This is why $\P(M)$ is called the \emph{projection lattice} of $M$. 
When $M$ is commutative, then $M$ can be identified with the $L_{\infty}$-space for some measure space. 
In this case, $\P(M)$ is identified with the collection of measurable sets modulo null sets. 
If we regard a general von Neumann algebra $M$ as a noncommutative measure space, then we may say that $\P(M)$ is the lattice of, so to speak, noncommutative measurable sets.
In this paper, we consider the following fairly natural question. 
\begin{q}
What is the general form of lattice isomorphisms between the projection lattices of a pair of von Neumann algebras?
\end{q}

Note that we only need to consider Question in the case where the two von Neumann algebras are of the same type (namely, one of I$_n$, I$_{\infty}$, II$_1$, II$_\infty$ or III), see \cite[Lemma 4.1]{M}.
In the finite-dimensional case, the fundamental theorem of projective geometry gives a partial answer to Question. 
It was von Neumann himself who was the first to consider Question for non-type I von Neumann algebras. 
He proved, as a byproduct of his theory on complemented modular lattices and von Neumann regular rings, that lattice isomorphisms between the projection lattices of a pair of type II$_1$ factors are in one-to-one correspondence with ring isomorphisms between the algebras of affiliated operators, see \cite[Chapter 2.4]{N} and Appendix 2 of \cite[Chapter 2.2]{N}. 

Feldman \cite{Fe} and Dye \cite{Dy} studied lattice isomorphisms of projection lattices under the orthogonality preserving assumption.  
Fillmore and Longstaff studied lattice isomorphisms for the lattice of closed subspaces of a complex normed space \cite{FL}, which particularly answers our Question in the case of type I$_\infty$ factors. 
In \cite{M}, the author studied lattice isomorphisms in the general setting of von Neumann algebras, and gave the following result which extends von Neumann's result for type II$_1$ factors. 
\begin{theorem}[{\cite[Theorem A]{M}}]\label{corr}
Let $M$ and $N$ be two von Neumann algebras.  
Suppose that $M$ does not have type I$_1$ nor I$_2$ direct summands, and that $\Phi\colon \P(M)\to \P(N)$ is a lattice isomorphism. 
Then there exists a unique ring isomorphism $\Psi\colon LS(M)\to LS(N)$ such that $\Phi(l(x)) = l(\Psi(x))$ for all $x\in LS(M)$. 
\end{theorem}
Here $LS(M)$ is the collection of locally measurable operators for a von Neumann algebra $M$, and $l(x)$ means the left support projection for $x\in LS(M)$.
The detailed definition of undefined terms will be given in Section \ref{preliminaries}.
Note that the converse of this theorem is easy to prove. 
Namely, if $\Psi\colon LS(M)\to LS(N)$ is a ring isomorphism, then the formula $\Phi(l(x)) = l(\Psi(x))$, $x\in LS(M)$, uniquely determines a lattice isomorphism $\Phi\colon \P(M)\to\P(N)$ \cite[Proposition 3.1]{M}. 
Therefore, our Question can be answered once we get the description of ring isomorphisms of locally measurable operator algebras. 

The author further studied such ring isomorphisms in \cite{M} in several cases. 
Below we consider a ring isomorphism $\Psi\colon LS(M)\to LS(N)$ under the assumption that $M$ and $N$ are of the same type. 

If $M$ and $N$ are of type I$_n$ ($1\leq n<\infty$), the description can be done utilizing ring isomorphisms between their centers \cite[Proposition 4.2]{M} (see also \cite{AAKD}). 
The center of $LS(M)$ is in fact identified with the algebra of all complex-valued measurable functions on a measure space (modulo the almost everywhere equivalence relation).   
Getting a further concrete description of the general form of the latter ring isomorphisms seems to be out of the reach of the theory of operator algebras (see the remark right after \cite[Proposition 4.2]{M}).

The author proved that $\Psi\colon LS(M)\to LS(N)$ is similar to a real $^*$-isomorphism whenever $M$ and $N$ are type I$_\infty$ or III von Neumann algebras \cite[Theorem B]{M}.
See \cite{M} for a more detailed overview of the above results and the history of related topics.
In the same paper, the author left the case of type II as an open question, conjecturing that every ring isomorphism $\Psi\colon LS(M)\to LS(N)$ is similar to a real $^*$-isomorphism when $M$ and $N$ are of type II \cite[Conjecture 5.1]{M}.
This conjecture was soon confirmed by Ayupov and Kudaybergenov in the case of type II$_1$ von Neumann algebras \cite[Theorem 1.4]{AK} . 

The purpose of this article is to fill in the still blank case of type II$_\infty$. 
Namely, the main theorem is as follows. 
\begin{theorem}\label{main}
Let $M, N$ be von Neumann algebras of type II$_{\infty}$. 
If $\Psi\colon LS(M)\to LS(N)$ is a ring isomorphism, then there exist a real $^*$-isomorphism $\psi\colon M\to N$ (which extends to a real $^*$-isomorphism from $LS(M)$ onto $LS(N)$) and an invertible element $y\in LS(N)$ such that  $\Psi(x)=y\psi(x)y^{-1}$, $x\in LS(M)$.
\end{theorem}

This together with \cite[Theorem B]{M} and \cite[Theorem 1.4]{AK} verifies
\begin{theorem}
Let $M, N$ be von Neumann algebras without type I$_\mathrm{fin}$ direct summands. 
If $\Psi\colon LS(M)\to LS(N)$ is a ring isomorphism, then there exist a real $^*$-isomorphism $\psi\colon M\to N$ (which extends to a real $^*$-isomorphism from $LS(M)$ onto $LS(N)$) and an invertible element $y\in LS(N)$ such that  $\Psi(x)=y\psi(x)y^{-1}$, $x\in LS(M)$.
\end{theorem}

Furthermore, applying \cite[Theorem A]{M}, we obtain 
\begin{theorem}\label{projection}
Let $M, N$ be von Neumann algebras without type I$_\mathrm{fin}$ direct summands. 
If $\Phi\colon \P(M)\to \P(N)$ is a lattice isomorphism, then there exist a real $^*$-isomorphism $\psi\colon M\to N$ and an invertible element $y\in LS(N)$ such that  $\Phi(p)=l(y\psi(p))$, $p\in \P(M)$.
\end{theorem}

The proof of our main theorem rests on the results in the previous papers \cite{M, AK} together with delicate arguments concerning unbounded operators and their invertibility.

\section{Preliminaries}\label{preliminaries}
In this paper, we will freely use basic facts about von Neumann algebras as in \cite{KR}.

For a Hilbert space $H$, the von Neumann algebra of all bounded linear operators on $H$ is written as $B(H)$.
Let $M\subset B(H)$ be a von Neumann algebra. 
We use the symbol $\sim$ (resp.\ $\prec$) to mean the Murray--von Neumann equivalence relation (resp.\ preorder relation) on $\P(M)$. 
That is, for $p, q\in \P(M)$, $p\sim q$ (resp.\ $p\prec q$) means that there exists a partial isometry $v\in M$ such that $p= vv^*$ and $q= v^*v$ (resp.\ $p\leq vv^*$ and $q= v^*v$). 
For $p, q\in \P(M)$, we write $p\perp q$ if $p$ and $q$ are orthogonal, i.e., $pq=qp=0$, or equivalently, $pH\perp qH$ in the Hilbert space $H$. 
We use the symbol $p^{\perp}:= 1-p$ for $p\in \P(M)$. 
The symbol $\Z(M)=\{x\in M\mid xy=yx\text{ for all }y\in M\}$ means the center of $M$.
For an operator $x\in M$, let $z(x)$ denote the central support of $x$, which is the minimal central projection $r$ in $\P(\Z(M))$ with $rx=x$.

Let us recall some terms about the geometry of $\P(M)$. 
A projection $p\in \P(M)$ is said to be \emph{finite} if any subprojection $p\geq q\in \P(M)$ with $p\sim q$ satisfies $p=q$. 
A projection $p\in \P(M)$ is said to be \emph{properly infinite} if there is $p\geq q\in \P(M)$ with $p\sim q\sim p-q$.  
It is well-known that for every $p\in \P(M)$ there is a unique subprojection $p\geq p_1\in \P(M)$ such that $p_1$ is finite, $p-p_1$ is properly infinite, and $z(p_1)z(p-p_1)=0$.
We call $p_1$ the \emph{finite part} of $p$, and $p-p_1$ the \emph{properly infinite part} of $p$.

\subsection{Various isomorphisms of $^*$-algebras}
For $^*$-algebras $A$ and $B$, a (not necessarily linear) bijection $\psi\colon A\to B$ is called
\begin{itemize}
\item a \emph{ring isomorphism} if it is additive and multiplicative,
\item a \emph{real $^*$-isomorphism} if it is real-linear, multiplicative, and satisfies $\psi(x^*)=\psi(x)^*$ for any $x\in A$, 
\item a \emph{$^*$-isomorphism} if it is a complex-linear real $^*$-isomorphism, and
\item a \emph{conjugate-linear $^*$-isomorphism} if it is a conjugate-linear real $^*$-isomorphism. 
\end{itemize}
It is well-known that a real $^*$-isomorphism between two von Neumann algebras is decomposed into the direct sum of a (complex-linear) $^*$-isomorphism and a conjugate-linear $^*$-isomorphism. See for example \cite[Lemma 2.1 (4)]{M}.
Two mappings $\psi_1$, $\psi_2$ from $A$ into $B$ are said to be \emph{similar} if there is an invertible element $y\in B$ such that $\psi_2(x)=y\psi_1(x)y^{-1}$ for every $x\in A$. 

In this paper, we will frequently use the following fact. 
Let $A_1, A_2, A_3$ be $^*$-algebras. 
If $\psi_1\colon A_1\to A_2$ and $\psi_2\colon A_2\to A_3$ are bijections that are similar to real $^*$-isomorphisms, then $\psi_2\circ\psi_1\colon A_1\to A_3$ is also similar to a real $^*$-isomorphism. The proof is easy.

\subsection{The algebra of locally measurable operators}\label{LSdef}
Let $M\subset B(H)$ be a von Neumann algebra. 
Let $x$ be a densely-defined closed operator on $H$. 
Then $x$ is said to be \emph{affiliated with $M$} (and we write $x\eta M$) if 
$yx\subset xy$ for any $y\in M'$.
Here $M':=\{y\in B(H)\mid ay=ya\text{ for any }a\in M\}$ denotes the commutant of $M$. 
We say $x\eta M$ is a \emph{measurable} operator of $M$ if the spectral projection $\chi_{(c, \infty)}(\lvert x\rvert)\in \P(M)$ is a finite projection in $M$ for some real number $c>0$. 
We say $x\eta M$ is a \emph{locally measurable} operator of $M$ if there exists an increasing sequence $(r_n)_{n\geq 1}$ of central projections in $M$ such that $r_n\nearrow 1$ and $xr_n$ is a measurable operator of $M$ for any $n$.
In other words, $x\eta M$ is locally measurable if the central support of the properly infinite part of the projection $\chi_{(c, \infty)}(\lvert x\rvert)$ tends to zero as $c\to\infty$.

We use the symbol $LS(M)$ to mean the collection of all locally measurable operators of $M$.
If $x, y\in LS(M)$, then $x^*$ and the closures of $xy$, $x+y$ are in $LS(M)$. 
More generally, the closure of any algebraic combination of elements in $LS(M)$ lies in $LS(M)$.  
These facts enable us to regard $LS(M)$ as a $^*$-algebra that contains $M$. 
In what follows, an algebraic combination of operators in $LS(M)$ refers to the corresponding operator in $LS(M)$ (i.e.\ its closure), unless otherwise stated.

It is not difficult to see that each element of $LS(M)$ can be written as the direct sum of bounded operators if $M$ is of type I or III. 
More precisely, for each $x\in LS(M)$, there exists an increasing sequence $(r_n)_{n\geq 1}$ of central projections in $M$ such that $r_n\nearrow 1$ and $xr_n$ is bounded for any $n$. 
Note that the same statement does not hold for a type II von Neumann algebra $M$.
See \cite{Y} and \cite{Se} for more details of (locally) measurable operators. 

For $x\in LS(M)$, let $l(x)\in \P(M)$ denote the left support of $x$. 
That is, $l(x) :=\bigwedge \{p\in \P(M)\mid px = x\}$. 
Similarly, for the right support, we use the symbol $r(x) :=\bigwedge \{p\in \P(M)\mid x = xp\}$.
Then the condition $xy=0$ is equivalent to $r(x)l(y)=0$ for a pair of operators $x, y\in LS(M)$, see the last part of \cite[Subsection 2.2]{M}.

\subsection{Two projections}\label{halmos}
When we play with projections, it is often useful to look at two projections. 
The so-called Halmos's two projection theorem \cite{H} (see also \cite{BS}) describes the relative position of a general pair of projections.  
Let us briefly recall it from the viewpoint of von Neumann algebra theory. 

Let $M\subset B(H)$ be a von Neumann algebra and $p, q\in \P(M)$. 
Then we may find a unique pair of mutually orthogonal equivalent projections $e_1, e_2\in \P(M)$ with the following properties. 
\begin{itemize}
\item The Hilbert space $H$ is orthogonally decomposed as  
\begin{equation}\label{h}
H = (p\wedge q^{\perp}) H \oplus (p^{\perp} \wedge q) H \oplus (p\wedge q) H \oplus (p^{\perp} \wedge q^{\perp}) H \oplus (e_1+e_2)H.
\end{equation}
\item By choosing a suitable identification of $(e_1+e_2)M(e_1+e_2)$ with $\M_2(e_1Me_1)$, $p$ and $q$ are decomposed as
\[
p= 1\oplus 0\oplus 1\oplus 0\oplus 
\begin{pmatrix}
1&0\\
0&0
\end{pmatrix},\quad 
q= 0\oplus 1\oplus 1\oplus 0\oplus 
\begin{pmatrix}
a^2&ab \\
ab&b^2
\end{pmatrix}
\]
with respect to the Hilbert space decomposition \eqref{h}, where $a$ and $b$ are positive injective operators in $e_1Me_1$ such that $a^2+b^2=e_1$. 
\end{itemize}
Assume that $p\wedge q = 0$. 
We say $p$ is \emph{LS-orthogonal} to $q$ if the operator $b\in e_1Me_1$ is invertible in $LS(e_1Me_1)$.
More details can be found in \cite{M}.

\subsection{A few facts on locally measurable operators}
In this subsection, let $M$ be a von Neumann algebra.
We prove a few facts about the invertibility of positive operators in $LS(M)$ for later use. 

\begin{lemma}\label{inverse}
Let $a\in LS(M)$ be a positive and injective operator.
Then the following are equivalent. 
\begin{enumerate}
\item The operator $a$ is not invertible in $LS(M)$.
\item There are a strictly decreasing sequence $c_n$ of positive real numbers and a sequence $f_n\in \P(M)$ of pairwise equivalent nonzero projections with the following properties: (i) $c_n\to 0$ as $n\to\infty$. (ii) $f_n\leq \chi_{(c_{n+1}, c_n]}(a)$.
\item There is a sequence $f_n\in \P(M)$ of pairwise equivalent nonzero projections such that ($f_naf_n$ is bounded for each $n\geq 1$ and) $\lVert f_naf_n\rVert\to 0$. 
\end{enumerate}
\end{lemma}
\begin{proof}
The implication (2)$\Rightarrow$(3) is clear.

(1)$\Rightarrow$(2) What follows is a reproduction of part of \cite[Proof of Lemma 2.2]{Mor2}.
By decomposing $M$ into a direct sum, we may assume that $M$ is of one of the three types I, II, or III.

First, we consider the case of type I and III. 
Then the assumption that $a\in LS(M)$ is not invertible implies that $r:= \bigwedge_{n\geq1} z(\chi_{(0, 1/n]}(a))\neq 0$. 
Considering the pair $(Mr, ar)$ instead of $(M, a)$, we may assume $\bigwedge_{n\geq1} z(\chi_{(0, 1/n]}(a))=1$. 
Take a normal (tracial) state $\tau$ on $\Z(M)$. 
We may also assume $\operatorname{supp}(\tau)=1 \in \P(\Z(M))$. 
By our assumptions, we may take a strictly decreasing sequence $(c_n)_{n\geq1}$ of positive real numbers that satisfies $c_n\to 0$ ($n\to \infty$) and $\tau(z(\chi_{(c_{n+1}, c_n]}(a))) \geq 1-3^{-n}$, $n\geq 1$. 
Then $\tau(\bigwedge_{n\geq1} z(\chi_{(c_{n+1}, c_n]}(a))) \geq 1-\sum_{n\geq 1}3^{-n}>0$ and thus $\bigwedge_{n\geq1} z(\chi_{(c_{n+1}, c_n]}(a)) \neq 0$. 
We may assume $\bigwedge_{n\geq1} z(\chi_{(c_{n+1}, c_n]}(a))=1$. 

By basic properties of projections, we obtain the following:
If $M$ is of type I, we can take an abelian projection $f_n\leq \chi_{(c_{n+1}, c_n]}(a)$ with $z(f_n)=1$ for each $n\geq 1$. 
If $M$ is of type III, by the assumption that $\Z(M)$ has a normal faithful state,  we can take a countably decomposable projection $f_n\leq \chi_{(c_{n+1}, c_n]}(a)$ with $z(f_n)=1$ for each $n\geq 1$. 
In both cases, using \cite[Corollary 6.3.5]{KR} and \cite[Proposition 6.4.6(iii)]{KR}, we see that $(f_n)_{n\geq1}$ is a family of mutually orthogonal equivalent nonzero projections satisfying $f_n\leq \chi_{(c_{n+1}, c_n]}(a)$ for each $n\geq 1$. 

Next, we consider the case where $M$ is of type II. 
For a projection $p\in \P(M)$, let $z_\infty(p)\in\P(\Z(M))$ denote the central support of the properly infinite part of $p$. 
Assume that $a\in LS(M)$ is not invertible in $LS(M)$. 
We may assume that $\bigwedge_{n\geq1} z_{\infty}(\chi_{(0, 1/n]}(a))=1$ without loss of generality. 
Then $M$ is of type II$_\infty$.
Take a normal semifinite faithful tracial weight $\tau$ on $M$ and a (finite) projection $p\in \P(M)$ with $\tau(p)=1$. 
It follows that $\chi_{(0, 1/n]}(a) \succ p$ for every $n\geq 1$.
By \cite[Lemma 2.3]{Mor2}, there exist a strictly decreasing sequence $(c_n)_{n\geq 1}$ of positive real numbers and a sequence $(p_n)_{n\geq 1}$ of projections in $M$ such that $c_n\to 0\,\,(n\to \infty)$, $p_n\leq \chi_{(c_{n+1}, c_n]}(a)$, $p_n\prec p$ and $\tau(p_n)\geq 1-3^{-n}$, $n\geq 1$. 
Take a projection $\tilde{p}_n\in \P(M)$ such that $p_n\sim \tilde{p}_n \leq p$, $n\geq 1$. 
Put $\tilde{p} := \bigwedge_{n\geq 1} \tilde{p}_n$. 
Then $\tau(\tilde{p}) = \tau(\bigwedge_{n\geq 1} \tilde{p}_n) \geq 1-\sum_{n\geq 1} 3^{-n}>0$. 
Hence $\tilde{p} \neq 0$.
Take a projection $f_n\in \P(M)$ such that $f_n\leq p_n$ and $f_n\sim\tilde{p}$, for $n\geq 1$. 
Then $(f_n)_{n\geq1}$ is a family of mutually orthogonal equivalent nonzero projections satisfying $f_n\leq \chi_{(c_{n+1}, c_n]}(a)$ for each $n\geq 1$.

(3)$\Rightarrow$(1) 
Assume that there is a sequence $f_n\in \P(M)$ of pairwise equivalent nonzero projections such that $\lVert f_naf_n\rVert\to 0$. 
We show that $a$ is not invertible in $LS(M)$. 
Define the function $F\colon [0, \infty)\to [0, 1]$ by $F(t)=\min\{t, 1\}$, $t\geq 0$.
Consider the bounded operator $F(a)$. 
Note that $F(a)\leq a$, hence $\lVert f_nF(a)f_n\rVert\leq \lVert f_naf_n\rVert$.
Moreover, it is easy to see that $a$ is invertible in $LS(M)$ if and only if $F(a)$ is invertible in $LS(M)$. 
Therefore, by considering $F(a)$ instead of $a$, we may assume that $a$ is a bounded operator. 

Assume for a contradiction that the bounded operator $a$ is invertible in $LS(M)$.  
Let $b$ be the inverse of $a$ in $LS(M)$.
Then there is an increasing sequence $r_m$ of central projections in $M$ such that $r_mb$ is measurable for each $m\geq 1$ and $r_m\nearrow 1$. 
We may take an integer $m\geq 1$ such that $r_mf_n\neq 0$ for some (or equivalently, every) $n\geq 1$. 
Let $\varepsilon >0$ be a positive real number. 
We may find $n\geq 1$ such that $\lVert f_naf_n \rVert\leq \varepsilon$.
It clearly follows that $f_n\wedge \chi_{(\varepsilon, \infty)}(a)=0$, which implies that $f_n\prec \chi_{(0, \varepsilon]}(a)$. 
(Here we imitated the argument in \cite[Proof of Proposition 2.2]{FK}.)
It follows that $0\neq r_m f_1\sim r_mf_n\prec \chi_{(0, \varepsilon]}(r_ma)$ for every positive real number $\varepsilon>0$. 
On the other hand, since $r_ma$ has an inverse $r_mb$ in $LS(r_mM)$ and $r_mb$ is measurable, we see that $\chi_{(0, c_0]}(r_ma)$ is a finite projection for some $c_0>0$.
We also have $\chi_{(0, c]}(r_ma)\to 0$ as $c\to 0$. 
Thus we arrive at a contradiction via the following statement: 
If a decreasing sequence of finite projections $p_n$ with $p_n\searrow 0$ and a projection $p$ in a von Neumann algebra $M$ satisfy $p\prec p_n$ for every $n$, then $p=0$. 
This statement can be obtained e.g.\ by considering the center-valued trace of $p$ in the finite von Neumann algebra $(p\vee p_1)M(p\vee p_1)$. Indeed, by the assumption the center-valued trace of $p$ is $0$, thus $p=0$.
\end{proof}

\begin{corollary}\label{LS}
Let $y\eta M$ be a positive self-adjoint operator. 
Then $y\notin LS(M)$ if and only if there are a strictly increasing sequence $c_n$ of positive real numbers and a sequence $f_n\in \P(M)$ of pairwise equivalent nonzero projections with the following properties: (i) $c_n\to\infty$ as $n\to\infty$. (ii) $f_n\leq \chi_{[c_n, c_{n+1})}(y)$. 
\end{corollary}
\begin{proof}
Apply Lemma \ref{inverse} (1)$\Leftrightarrow$(2) to $a:= (1+y)^{-1}$, which is a positive injective (bounded) operator in $M$.
\end{proof}

\begin{corollary}\label{ortho}
Let $p, q\in \P(M)$ satisfy $p\wedge q=0$.
Then $p$ is not LS-orthogonal to $q$ if and only if there is a sequence $f_n\in\P(M)$ of pairwise equivalent nonzero subprojections of $p$
such that $\Vert f_n q^\perp f_n\rVert\to 0$.
\end{corollary}
\begin{proof}
Applying Subsection \ref{halmos} and using the same symbols as there, we may regard 
\[
p= 1\oplus 0\oplus 0\oplus 
\begin{pmatrix}
1&0\\
0&0
\end{pmatrix},\quad 
q= 0\oplus 1\oplus 0\oplus 
\begin{pmatrix}
a^2&ab \\
ab&b^2
\end{pmatrix}, 
\]
and thus 
\[
q^\perp 
= 1\oplus 0\oplus 1\oplus 
\begin{pmatrix}
b^2&-ab \\
-ab&a^2
\end{pmatrix}.
\]
Recall that $p$ is not LS-orthogonal to $q$ if and only if $b\in e_1Me_1$ is not invertible in $LS(e_1Me_1)$. 
Clearly, this is equivalent to the condition $1\oplus b^2\in (p\wedge q^\perp+e_1)M(p\wedge q^\perp+e_1)$ is not invertible in $LS((p\wedge q^\perp+e_1)M(p\wedge q^\perp+e_1))$.   
Now apply Lemma \ref{inverse} (1)$\Leftrightarrow$(3) to obtain the desired conclusion.
\end{proof}

\section{Proof of the main theorem}

\subsection{Reduction to a special case}\label{reduce}
First, let us recall a basic fact about a type II$_\infty$ von Neumann algebra. 
We give a sketch of a proof for completeness. 

\begin{lemma}\label{31}
Let $M$ be a type II$_\infty$ von Neumann algebra. 
Let $p\in \P(M)$ be a nonzero finite projection. 
Then there is a family $Z$ of pairwise orthogonal nonzero central projections with sum $z(p)$ satisfying the following properties: 
For each $r\in Z$, there is a family $(q_i)_{i\in I}$ of mutually orthogonal equivalent projections such that $rp\in \{q_i\mid i\in I\}$ and $\sum_{i\in I}q_i=r$. 
\end{lemma}
\begin{proof}
Take a maximal family of mutually orthogonal projections $(p_i)_{i\in I}$ with $p_i\sim p$ for each $i\in I$ and $p\in \{p_i\mid i\in I\}$. 
By the comparison theorem, we may find a central projection $r\in \P(\Z(M))$ such that $(1-\bigvee_{i\in I}p_i)r\prec pr$ and $(1-\bigvee_{i\in I}p_i)r^\perp\succ pr^\perp$. 
The maximality of $(p_i)_{i\in I}$ implies that $r\neq 0$.  
Note that $I$ is an infinite set since $Mr$ is also a type II$_\infty$ von Neumann algebra.
Thus we may retake a family of mutually orthogonal equivalent projections $(q_i)_{i\in I}$ with the properties $rp\in \{q_i\mid i\in I\}$ and $\sum_{i\in I}q_i=r$. 
(This retaking can be done in the following manner: Let $i_0\in I$ be the element with $p_{i_0}=p$. Fix an element $i_0\neq j\in I$, and take a bijection $\sigma\colon I\setminus \{i_0, j\}\to I\setminus\{i_0\}$. For each $i\in I\setminus\{i_0, j\}$, we may take a projection $\check{p}_i\leq p_i$ with $\check{p}_ir\sim (1-\bigvee_{i\in I}p_i)r$ since $(1-\bigvee_{i\in I}p_i)r\prec pr\sim p_ir$. Then we may set $q_i=(p_i-\check{p}_i +\check{p}_{\sigma(i)})r$ for each $i\in I\setminus \{i_0, j\}$, $q_{i_0}=pr=rp$ and $q_j=(p_j-\check{p}_j+1-\bigvee_{i\in I}p_i)r$.)
Now an application of Zorn's lemma leads us to the desired conclusion.
\end{proof}

Let $M, N$ be von Neumann algebras of type II$_\infty$, and $\Psi\colon LS(M)\to LS(N)$ a ring isomorphism. 
The goal of this article is to show that $\Psi$ is similar to a real $^*$-isomorphism. 
The aim of this subsection is to reduce our task to a special case.
Note that $\Psi$ maps a central projection in $M$ to a central projection in $N$ because an idempotent in the center of $LS(M)$ is a projection.
First, we show the following claim. 
\begin{claim}\label{claim1}
There is a family of mutually orthogonal central projections $Z\subset \P(\Z(M))$ with $\sum_{r\in Z}r=1$ satisfying the following properties: 
For each $r\in Z$, there are von Neumann algebras $M_0, N_0$ of type II$_1$, Hilbert spaces $H, K$, $^*$-isomorphisms $\pi\colon rM\to M_0\otimesbar B(H)$,  $\rho\colon \Psi(r)N\to N_0\otimesbar B(K)$, and projections $e_0\in \P(B(H))$, $f_0\in \P(B(K))$ of rank one such that $l(\rho\circ \Psi\circ \pi^{-1}(1\otimes e_0))=1\otimes f_0$.
\end{claim}
\begin{proof}
Recall that the lattice isomorphism $\Phi\colon \P(M)\to\P(N)$ determined by $\Phi(p)=l(\Psi(p))$, $p\in \P(M)$, maps a finite projection to a finite projection (see \cite[Proof of Lemma 4.1]{M}). 
Fix a finite projection $p\in \P(M)$ with $z(p)=1$. 
Applying the preceding lemma to $p\in \P(M)$ and $\Phi(p)\in \P(N)$, we obtain families $Z_1\subset \P(\Z(M))$, $Z_2\subset \P(\Z(N))$ of mutually orthogonal central projections with sum $1$ in $M$ and $N$, respectively. 
Set $Z=\{z\Psi^{-1}(w)\mid z\in Z_1, w\in Z_2\}$.
For each element $r=z\Psi^{-1}(w)\in Z$ with $z\in Z_1$, $w\in Z_2$, we see that $\Psi$ restricts to a ring isomorphism from $rLS(M)=LS(rM)$ onto $\Psi(r)LS(N)=LS(\Psi(r)N)$. 
Observe that $M_0:= prMp$ and $N_0:= \Phi(p)\Psi(r)N\Phi(p)$ are von Neumann algebras of type II$_1$. 
By Lemma \ref{31}, we may take $^*$-isomorphisms $\pi\colon rM\to M_0\otimesbar B(H)$,  $\rho\colon \Psi(r)N\to N_0\otimesbar B(K)$ such that $\pi(pr)=1\otimes e_0$, $\rho(\Phi(p)\Psi(r))=1\otimes f_0$ for some projections $e_0\in \P(B(H))$, $f_0\in \P(B(K))$ of rank one. Then we have 
\[
l(\rho\circ \Psi\circ \pi^{-1}(1\otimes e_0))= l(\rho\circ \Psi(pr))=\rho(l(\Psi(pr)))=\rho(\Phi(pr)).
\]
Since $r$ is central we see that $\Phi(pr)=\Phi(p)\Psi(r)$, and we obtain $\rho(\Phi(pr))=\rho(\Phi(p)\Psi(r))=1\otimes f_0$.
\end{proof}

In order to show Theorem \ref{main}, it suffices to show that the restriction of $\Psi$ to $rLS(M)=LS(rM)$ is similar to a real $^*$-isomorphism for each $r\in Z$ in Claim \ref{claim1}.
Therefore, refreshing the notations, we just need to consider the following situation. 
Assume that $M_0, N_0$ are von Neumann algebras of type II$_1$, $H, K$ are infinite-dimensional Hilbert spaces, and that $e_0\in \P(B(H))$, $f_0\in \P(B(K))$ are projections of rank one. 
Assume further that $\Psi\colon LS(M) \to LS(N)$ is a ring isomorphism satisfying $l(\Psi(p_0))=q_0$, where $M:= M_0\otimesbar B(H)$, $N_0:=N_0\otimesbar B(K)$, and $p_0:=1\otimes e_0$, $q_0:=1\otimes f_0$. 
Under this setting, we have the following claim. 

\begin{claim}\label{claim2}
There is a mapping $\psi\colon LS(N_0\otimesbar B(K))\to LS(M_0\otimesbar B(K))$ that is similar to a real $^*$-isomorphism with $\psi\circ\Psi(x\otimes e_0)=x\otimes f_0$ for every $x\in LS(M_0)$.
\end{claim}


\begin{proof}

We see that the operator $\Psi(p_0)$ is an idempotent with $l(\Psi(p_0)) =q_0$. 
Thus we may take an invertible operator $y_1\in LS(N)$ such that $y_1\Psi(p_0)y_1^{-1}=q_0$. 
(Indeed, we may set $y_1:= \Psi(p_0)+q_0^\perp$. 
Note that $0=\Psi(p_0^\perp p_0)=\Psi(p_0^\perp)\Psi(p_0)$, which implies $0=\Psi(p_0^\perp)l(\Psi(p_0))=(1-\Psi(p_0))q_0$. 
Using this equality, we may easily obtain $y_1^{-1}=-\Psi(p_0)+1+q_0$ and $y_1\Psi(p_0)y_1^{-1}=q_0$.) 
Then the mapping $\Psi_1=y_1\Psi(\cdot)y_1^{-1}$ satisfies $\Psi_1(p_0)=q_0$. 

From here we apply the result by Ayupov and Kudaybergenov. 
Observe that now $\Psi_1$ restricts to a ring isomorphism from $p_0LS(M)p_0$ onto $q_0LS(N)q_0$. 
Let us identify $p_0LS(M)p_0$ and $q_0LS(N)q_0$ with $LS(M_0)$ and $LS(N_0)$, respectively, in a natural manner. 
By \cite[Theorem 1.4]{AK}, we may find  a real $^*$-isomorphism $\psi_0\colon M_0\to N_0$ (which also determines a real $^*$-isomorphism from $LS(M_0)$ onto $LS(N_0)$) and an invertible operator $y_0\in LS(N_0)$ such that $\Psi_1(x)=y_0\psi_0(x)y_0^{-1}$ for every $x\in LS(M_0)$. 
Note that $\psi_0\colon M_0\to N_0$ determines a real $^*$-isomorphism $\psi_1\colon M_0\otimesbar B(K)\to N_0\otimesbar B(K)$ such that $\psi_1(a\otimes b)=\psi_0(a)\otimes b$ for any $a\in M_0$, $b\in B(K)$ (which in turn determines a real $^*$-isomorphism from $LS(M_0\otimesbar B(K))$ onto $LS(N_0\otimesbar B(K))$). 
Put $y_{00} := y_0\otimes e_0 + 1\otimes e_0^{\perp}\in LS(M_0\otimesbar B(K))$. 
Then the mapping $LS(M_0\otimesbar B(H))\ni x \mapsto \psi_1^{-1}(y_{00}^{-1}\Psi_1(x)y_{00})\in LS(M_0\otimesbar B(K))$ sends $x\otimes e_0$ to $x\otimes f_0$ for every $x\in LS(M_0)$, so the claim is settled.
\end{proof}
Then our task reduces to considering $\psi\circ\Psi$ instead of $\Psi$ in Claim \ref{claim2}.
Let us again refresh the notations. 
We need to consider a von Neumann algebra $M_0$ of type II$_1$, infinite-dimensional Hilbert spaces $H, K$, projections $e_0\in \P(B(H))$, $f_0\in \P(B(K))$ of rank one, and a ring isomorphism $\Psi\colon LS(M_0\otimesbar B(H))\to LS(M_0\otimesbar B(K))$ satisfying $\Psi(x\otimes e_0)=x\otimes f_0$ for every $x\in LS(M_0)$.
By considering $\Psi^{-1}$ instead of $\Psi$ if necessary, we may assume that $\dim K\leq \dim H$. 
Fix a linear isometry $w$ from $K$ into $H$ with $wf_0w^*=e_0$. 
By the embedding $x\mapsto (1\otimes w)x(1\otimes w^*)$, we may identify $LS(M_0\otimesbar B(K))$ with $LS(M_0\otimesbar B(wK))$, which may be regarded as a subalgebra of $LS(M_0\otimesbar B(H))$. 
Note that $(1\otimes w)q_0(1\otimes w^*)=p_0$. 
To consider Theorem \ref{main}, it suffices to work with the map $LS(M_0\otimesbar B(H))\ni x\mapsto (1\otimes w)\Psi(x)(1\otimes w^*)\in LS(M_0\otimesbar B(wK))$ instead of $\Psi$. 

\color{black}

\subsection{Proof of the main theorem}
By the preceding subsection, to verify Theorem \ref{main}, we just need to prove the following statement. We refresh our notations once again.
\begin{proposition}\label{special}
Let $L$ be a Hilbert space and $M_0\subset B(L)$ a von Neumann algebra of type II$_1$. 
Let $H$ be an infinite-dimensional Hilbert space, $K\subset H$ an infinite-dimensional closed subspace, and let $e_0\in \P(B(K))$ be a rank-one projection. 
Set $M:= M_0\otimesbar B(H)$ and $N:=M_0\otimesbar B(K)$.
Put $p_0:=1\otimes e_0\in \P(N) (\subset \P(M))$. 
Let $\Psi\colon LS(M)\to LS(N)$ be a ring isomorphism such that $\Psi(x\otimes e_0)=x\otimes e_0$ for each $x\in LS(M_0)$. 
Then $\Psi$ is similar to a real $^*$-isomorphism.
\end{proposition}

From now on we give a proof of Proposition \ref{special}. Let $\Phi\colon \P(M)\to \P(N)$ denote the lattice isomorphism associated with $\Psi$.
Namely, $\Phi\colon \P(M)\to \P(N)$ is the lattice isomorphism that is determined by the formula $\Phi(p)=l(\Psi(p))$, $p\in \P(M)$ (see \cite[Proposition 3.1]{M}).
We define an unbounded operator $\check{y}$ on the Hilbert space $L\otimes H$. 
We will show that it is closable and the closure brings us to the solution of our task.

Let $\P(M)_0$ denote the collection of all projections $p\in \P(M)$ with $p\sim p_0$. 
Let $\P(M)_1$ denote the collection of all projections $p\in \P(M)$ such that there are  $n\geq 1$ and mutually orthogonal projections $p_1, \ldots, p_n\in \P(M)_0$ with $p=p_1+\cdots+p_n$. 
To define $\check{y}$, we first construct a family of operators $(y_p)_{p\in \P(M)_1}\subset LS(M)$. 

Let $p=p_1+\cdots+p_n\in \P(M)_1$ with $p_1, \ldots, p_n\in \P(M)_0$. 
For each $i\in \{1, \ldots, n\}$, take a partial isometry $v_i\in M$ such that $v_i^*v_i=p_0$, $v_iv_i^*=p_i$. 
Because $\Psi$ is a ring isomorphism, we have $r(\Psi(x_1))\leq r(\Psi(x_2))$ for every pair $x_1, x_2\in LS(M)$ with $r(x_1)\leq r(x_2)$. (To show this, imitate the proof of \cite[Proposition 3.1]{M}.) It follows that $\Psi(v_i)\in LS(N)$ is an operator with $r(\Psi(v_i))=r(\Psi(p_0))=  p_0\in N\subset M$. 
Thus $\Psi(v_i)v_i^*\in LS(M)$ satisfies $r(\Psi(v_i)v_i^*)=p_i$. 
Put $y_p := \Psi(v_1)v_1^*+\cdots+\Psi(v_n)v_n^*\in LS(M)$.
\begin{claim}
The operator $y_p\in LS(M)$ is well-defined, and satisfies $r(y_p)=p, l(y_p)=\Phi(p)$ and $y_{p+q}=y_p+y_q$ for any pair $p, q\in \P(M)_1$ with $q\perp p$. 
\end{claim}
\begin{proof}
We first show that $y_p$ is well-defined. 
Let $p=p'_1+\cdots+p'_n$ with $p'_1, \ldots, p'_n\in \P(M)_0$ be another decomposition of $p$. 
Let $v'_i\in M$ satisfy ${v'_i}^*v'_i=p_0$, $v'_i{v'_i}^*=p'_i$ for $i\in \{1, \ldots, n\}$. 
For each $i\in \{1, \ldots, n\}$, we obtain
\[
\begin{split}
\Psi(v_i)v_i^* = \Psi(v_i)v_i^*p &= \Psi(v_i)v_i^* (v'_1{v'_1}^*+\cdots+v'_n{v'_n}^*)\\
&= \Psi(v_i)(v_i^*v'_1){v'_1}^*+\cdots+\Psi(v_i)(v_i^*v'_n){v'_n}^*\\
&= \Psi(v_iv_i^*v'_1){v'_1}^*+\cdots+\Psi(v_iv_i^*v'_n){v'_n}^*,
\end{split}
\]
where the third equality is a consequence of $v_i^*v'_j\in p_0Mp_0=\{x\otimes e_0\mid x\in M_0\}$ for each $j$.
Therefore, $\Psi(v_1)v_1^*+\cdots+\Psi(v_n)v_n^*$ is equal to 
\[
\sum_{1\leq i, j\leq n} \Psi(v_iv_i^*v'_j){v'_j}^*
= \sum_{j=1}^n \Psi\left(\sum_{i=1}^n v_iv_i^*v'_j\right){v'_j}^*
= \sum_{j=1}^n \Psi\left(pv'_j\right){v'_j}^*
= \sum_{j=1}^n \Psi(v'_j){v'_j}^*.
\]
This ensures that $y_p$ is well-defined. 
It clearly follows from the definition that $y_{p+q}=y_p+y_q$ for any pair $p, q\in \P(M)_1$ with $q\perp p$.

From the definition again, it is easy to see that $r(y_p)\leq p$ for any $p\in \P(M)_1$. 
Let $q\in \P(M)$ be any nonzero projection with $q\leq p$. 
We may find a nonzero projection $q_1\leq q$ with $q_1\prec p_0$. 
Then we may in turn take mutually orthogonal $p_1, \ldots, p_n\in \P(M)_0$ such that $q_1\leq p_1$ and $p=p_1+\cdots+p_n$ because $M$ is properly infinite and $p_0$ is finite. 
For each $i\in \{1, \ldots, n\}$, take a partial isometry $v_i\in M$ such that $v_i^*v_i=p_0$, $v_iv_i^*=p_i$. 
Then we obtain $y_p = \Psi(v_1)v_1^*+\cdots+\Psi(v_n)v_n^*$. 
It follows that $y_pq_1=\Psi(v_1)v_1^*q_1$, which is nonzero because $\Psi(v_1^*)\Psi(v_1)v_1^*q_1= \Psi(v_1^*v_1)v_1^*q_1=\Psi(p_0)v_1^*q_1=p_0v_1^*q_1=v_1^*q_1\neq 0$.  
Thus we obtain $y_pq\neq 0$ for any nonzero $q\leq p$, which leads to the equality $r(y_p)=p$. 

Let $p=p_1+\cdots+p_n\in \P(M)_1$ with $p_1, \ldots, p_n\in \P(M)_0$ again, and for each $i\in \{1, \ldots, n\}$, take a partial isometry $v_i\in M$ such that $v_i^*v_i=p_0$, $v_iv_i^*=p_i$. 
Then we have 
\[
l(y_p)\leq \bigvee_{1\leq i\leq n}l(\Psi(v_i))=\bigvee_{1\leq i\leq n}\Phi(l(v_i))= \bigvee_{1\leq i\leq n}\Phi(p_i) = \Phi\left(\bigvee_{1\leq i\leq n}p_i\right)=\Phi(p).
\] 
On the other hand, we see 
\[
l(y_p)\sim r(y_p)=p=p_1+\cdots+p_n = r(y_{p_1})+\cdots+r(y_{p_n})\succ l(y_{p_1})\vee\cdots\vee l(y_{p_n}).
\]
The last order relation $\succ$ is obtained by e.g.\ comparing the center-valued trace on the finite von Neumann algebra $eMe$ for $e=(r(y_{p_1})+\cdots+r(y_{p_n}))\vee l(y_{p_1})\vee\cdots\vee l(y_{p_n})$.
Since $l(y_{p_i})=l(\Psi(v_i)v_i^*)=l(\Psi(v_i)v_i^*v_i)=l(\Psi(v_i)p_0)=l(\Psi(v_i))=\Phi(p_i)$ for each $i$, we obtain $l(y_p)\succ \bigvee_{1\leq i\leq n}\Phi(p_i)=\Phi(p)$. 
This together with the fact that $\Phi(p)$ is finite (see \cite[Proof of Lemma 4.1]{M}) implies that $l(y_p)=\Phi(p)$.
\end{proof}

Assume that $p, q\in \P(M)_1$ satisfy $p\leq q$. 
By the equations $y_q=y_p+y_{q-p}$, $r(y_p)=p$,  
and $r(y_{q-p})=q-p$,
we obtain the equality $y_qp=y_p$ in $LS(M)$. 
On the other hand, it is clear that the product $y_qp$ as an unbounded operator is already closed.  
Therefore, we have $\D(y_p)\cap p(L\otimes H)=  \D(y_q)\cap p(L\otimes H)$ and $y_ph=y_q h$ for every $h\in \D(y_p)\cap p(L\otimes H)$.

Let $\check{y}$ be the unbounded linear operator on $L\otimes H$ determined by the following conditions: 
The domain $\D(\check{y})$ is the union of $\D(y_p)\cap p(L\otimes H)$ for $p\in \P(M)_1$. 
(Here we consider each $y_p\in LS(M)$ as an unbounded closed operator on $L\otimes H$.)
If $h\in \D(y_p)\cap p(L\otimes H)$ for $p\in \P(M)_1$, then we define $\check{y}h:=y_ph$. 
Let $h\in L\otimes H$ satisfy $h\in\D(y_{p_i})\cap {p_i}(L\otimes H)$ for $p_i\in \P(M)_1$, $i=1, 2$. 
Then we may find a projection $q\in \P(M)_1$ with $p_1\vee p_2\leq q$. 
Thus the preceding paragraph implies that $y_{p_1}h=y_q h=y_{p_2}h$, and this verifies that $\check{y}$ is well-defined.
The same argument also shows that the equality $\check{y}p=y_p$ holds as unbounded operators for any $p\in \P(M)_1$.
The domain of $\check{y}$ contains dense subspace of the range of $p$ for every $p\in \P(M)_1$, so $\check{y}$ is densely-defined.

\begin{claim}
The operator $\check{y}$ is closable.
\end{claim}
\begin{proof}
It suffices to verify that if a sequence $h_n\in \D(\check{y})$ satisfies $h_n\to 0$ and $\check{y}h_n\to k$ for some $k\in L\otimes K$ in norm then $k=0$. 
We know that $\D(y_p)\cap p(L\otimes H)\subset \D(y_q)\cap q(L\otimes H)$ if $p, q\in \P(M)_1$ satisfy $p\leq q$.
Hence we may find a sequence of mutually orthogonal projections $(p_n)_{n\geq 1}\subset \P(M)_1$ such that $h_n\in \D(y_{p_1+\cdots+p_n})\cap (p_1+\cdots+p_n)(L\otimes H)$ for each $n\geq 1$.
Since $y_{p_1+\cdots+p_n} =y_{p_1+\cdots+p_n}(p_1+\cdots+p_n) = \overline{y_{p_1}+\cdots+y_{p_n}}$ as unbounded operators, we may assume with no loss of generality that there exist $h_{n, 1}\in \D(y_{p_1})\cap p_1(L\otimes H)$, \ldots, $h_{n, n}\in \D(y_{p_n})\cap p_n(L\otimes H)$ such that $h_n = h_{n, 1}+\cdots+h_{n, n}$. 
Note that $\check{y}h_{n,m}=y_{p_m}h_{n,m}\in \Phi(p_m)(L\otimes K)$ for $m\in \{1, \ldots, n\}$. 
Fix $j\geq 1$.
For $n\geq j$, we obtain
\[
\begin{split}
\Phi\left(\bigvee_{l\geq j+1}p_l\right)^{\perp}\check{y}h_n &= \Phi\left(\bigvee_{l\geq j+1}p_l\right)^{\perp}\check{y}(h_{n, 1}+\cdots+h_{n, j})\\
&= \Phi\left(\bigvee_{l\geq j+1}p_l\right)^{\perp}\check{y}(p_1+\cdots+p_j)h_n.
\end{split}
\]
Since $y_{p_1+\cdots+p_j}\in LS(M)$, the unbounded operator $\Phi\left(\bigvee_{l\geq j+1}p_l\right)^{\perp}\check{y}(p_1+\cdots+p_j)=\Phi\left(\bigvee_{l\geq j+1}p_l\right)^{\perp}y_{p_1+\cdots+p_j}$ is closable. 
On the other hand, we know $h_n\to 0$ and $\Phi\left(\bigvee_{l\geq j+1}p_l\right)^{\perp}\check{y}h_n\to\Phi\left(\bigvee_{l\geq j+1}p_l\right)^{\perp}k$.
Thus we obtain $\Phi\left(\bigvee_{l\geq j+1}p_l\right)^{\perp}k=0$. 
We have shown that $k$ lies in $\Phi\left(\bigvee_{l\geq j+1}p_l\right)(L\otimes K)$ for each $j\geq 1$. 
On the other hand, we know $\bigwedge_{j\geq1}\bigvee_{l\geq j+1}p_l =0$, which together with the fact that $\Phi$ is a lattice isomorphism shows that $\bigwedge_{j\geq1}\Phi\left(\bigvee_{l\geq j+1}p_l\right)=\Phi\left(\bigwedge_{j\geq1}\bigvee_{l\geq j+1}p_l\right) =0$. 
Thus we finally arrive at the conclusion $k=0$, and we see that $\check{y}$ is closable.
\end{proof}

Let $\hat{y}$ denote the closure of $\check{y}$.
Since $y_p\eta M$ for every $p\in \P(M)_1$, we may easily obtain $\hat{y} \eta M$. 
For each $p\in \P(M)_1$, we see that $y_p$ and $\hat{y}p$ are densely-defined closed operators with $y_p\subset \hat{y}p$. 
Since $p$ is finite, we obtain $y_p= \hat{y}p$. 
(This is a consequence of the fact that a closed densely-defined operator affiliated with a finite von Neumann algebra has no proper closed extension, see e.g.\ \cite[Exercise 6.9.54]{KR}.)

For a projection $p\in \P(M)$ with $p\prec p_0$, we may find a projection $p_1\in \P(M)_0$ with $p\leq p_1$. 
Fix a partial isometry $v_1\in M$ with $v_1^*v_1=p_0$ and $v_1v_1^*=p_1$. 
Since $y_{p_1}=\Psi(v_1)v_1^*$, we see that 
\[
l(\hat{y}p)=l(y_{p_1}p)=l(\Psi(v_1)v_1^*p)=l(\Psi(v_1)v_1^*pv_1)=l(\Psi(v_1v_1^*pv_1))=l(\Psi(pv_1))=\Phi(p).
\] 
Thus we obtain $l(\hat{y}p)=\Phi(p)$ for every $p\in \P(M)$ with $p\prec p_0$.

Let $\hat{y}=v\lvert \hat{y}\rvert=\lvert \hat{y}^*\rvert v$ be the polar decomposition of $\hat{y}$. 
Then $v$ is a partial isometry in $M$.
The fact that $r(y_p)=p$ and $l(y_p)=\Phi(p)$ for $p\in \P(M)_1$ shows that $v$ satisfies $v^*v=1\in M$ and $vv^*=1\otimes p_K\in M_0\otimesbar B(H)$, where $p_K\in \P(B(H))$ denotes the projection onto $K$.
Thus $v$ determines a $^*$-isomorphism $M\ni x\mapsto vxv^*\in N$.

\begin{claim}
$\lvert \hat{y}\rvert\in LS(M)$.
\end{claim}
\begin{proof}
Assume that $\lvert \hat{y}\rvert$ is not in $LS(M)$. 
By Corollary \ref{LS}, we may take a strictly increasing sequence $c_n$ of positive real numbers and a sequence $f_n\in\P(M)$ of pairwise equivalent nonzero projections with the following properties: (i) $c_n\to\infty$ as $n\to\infty$. (ii) $f_n\leq \chi_{[c_n, c_{n+1})}(\lvert \hat{y}\rvert)$. 
By taking subprojections of $f_n$ if necessary, we may additionally assume that $f_n\prec p_0$ for every $n\geq 1$. 
We may take a strictly increasing sequence of positive integers $(n_k)_{k\geq 1}\subset \mathbb{N}$ such that 
$c_{n_{2k+1}}/c_{n_{2k}+1}\to \infty$.
Take a sequence of partial isometries $v_k\in M$ such that $v_k^*v_k=f_{n_{2k}}$ and $v_kv_k^*=f_{n_{2k+1}}$ for each $k\geq 1$. 

Set $g_k:=(f_{n_{2k}}+v_k+v_k^*+f_{n_{2k+1}})/2$ and $e_k :=f_{n_{2k+1}}$.
Then $(g_k)_{k\geq 1}$, $(e_k)_{k\geq 1}$ are sequences of mutually orthogonal equivalent projections. 
Set $g:= \bigvee_{k\geq 1} g_k$ and $e:= \bigvee_{k\geq 1} e_k$. 
It is clear that $g$ and $e$ are mutually LS-orthogonal.
We claim that $\lVert l(\lvert \hat{y}\rvert g_k)l(\lvert \hat{y}\rvert e_k)^\perp l(\lvert \hat{y}\rvert g_k)\rVert\to 0$ as $k\to \infty$.
To see this, observe that each unit vector of the range of the operator $\lvert \hat{y}\rvert g_k$ is of the form $\lvert\hat{y}\rvert h+\lvert\hat{y}\rvert v_kh$ for some $h\in f_{n_{2k}} (L\otimes H)$. 
Since $\lvert\hat{y}\rvert h=\lvert\hat{y}\rvert \chi_{[c_{n_{2k}}, c_{n_{2k}+1})}(\lvert\hat{y}\rvert)h$ and $\lvert\hat{y}\rvert v_kh=\lvert\hat{y}\rvert \chi_{[c_{n_{2k+1}}, c_{n_{2k+1}+1})}(\lvert\hat{y}\rvert)v_kh$, we obtain $\lvert\hat{y}\rvert h\perp \lvert\hat{y}\rvert v_kh$, $\lVert \lvert\hat{y}\rvert h\rVert\leq c_{n_{2k}+1} \rVert h\rVert$, and $\lVert \lvert\hat{y}\rvert v_kh\rVert\geq c_{n_{2k+1}} \rVert h\rVert$. 
Therefore, if $c_{n_{2k+1}}/c_{n_{2k}+1}$ is large, then the unit vector $\lvert\hat{y}\rvert h+\lvert\hat{y}\rvert v_kh$ is close to the unit vector $\lvert\hat{y}\rvert v_kh/\lVert \lvert\hat{y}\rvert v_kh\rVert$ which lies in the range of $\lvert \hat{y}\rvert e_k$. 
This together with an application of Subsection \ref{halmos} to the pair $l(\lvert \hat{y}\rvert g_k), l(\lvert \hat{y}\rvert e_k)$ shows that $\lVert l(\lvert \hat{y}\rvert g_k)l(\lvert \hat{y}\rvert e_k)^\perp l(\lvert \hat{y}\rvert g_k)\rVert\to 0$ as $k\to\infty$.

Since $g_k\prec p_0$, we obtain $l(\lvert \hat{y}\rvert g_k) = v^*l(v\lvert \hat{y}\rvert g_k)v=v^*\Phi(g_k)v$, and similarly, $l(\lvert \hat{y}\rvert e_k) =v^*\Phi(e_k)v$ for each $k$.
Therefore, we obtain 
\[\begin{split}
\lVert \Phi(g_k)\Phi(e_k)^\perp \Phi(g_k)\rVert&=\lVert v^*\Phi(g_k)v\cdot (v^*\Phi(e_k) v)^\perp\cdot v^*\Phi(g_k)v\rVert\\
&=\lVert l(\lvert \hat{y}\rvert g_k)l(\lvert \hat{y}\rvert e_k)^\perp l(\lvert \hat{y}\rvert g_k)\rVert\to 0.
\end{split}
\] 
Since $\Phi(e)\geq \Phi(e_k)$, we obtain 
\[
0\leq \Phi(g_k)\Phi(e)^\perp\Phi(g_k) \leq \Phi(g_k)\Phi(e_k)^\perp\Phi(g_k),
\]
which shows that $\lVert \Phi(g_k)\Phi(e)^\perp\Phi(g_k)\rVert\to 0$ as $k\to \infty$.
Note that $(\Phi(g_k))_{k\geq 1}$ is a sequence of pairwise equivalent subprojections of $\Phi(g)$. 
(Indeed, for a pair $i\neq j$ of positive integers, we may find a partial isometry $w\in M$ such that $w^*w=g_i$ and $ww^*=g_j$. Then $g' =(g_i+w+w^*+g_j)/2+(1-g_i-g_j) \in \P(M)$ satisfies $g_i\wedge g'=g_j\wedge g'=0$ and $g_i\vee g'=g_j\vee g'=1$. It follows that $\Phi(g_i)\wedge \Phi(g')=\Phi(g_j)\wedge \Phi(g')=0$ and $\Phi(g_i)\vee \Phi(g')=\Phi(g_j)\vee \Phi(g')=1$. Subsection \ref{halmos} applied to the pair $\Phi(g_i), \Phi(g')$ implies $\Phi(g_i)\sim \Phi(g')^\perp$, and similarly, we obtain $\Phi(g_j)\sim \Phi(g')^\perp$, thus $\Phi(g_i)\sim\Phi(g_j)$.)
Moreover, by \cite[Lemma 3.6]{M}, the projections $\Phi(g)$, $\Phi(e)$ in $N$ are LS-orthogonal.
This contradicts Corollary \ref{ortho}.  
Therefore, we arrive at the conclusion that $\lvert \hat{y}\rvert\in LS(M)$. 
\end{proof}

By applying Lemma \ref{inverse} and imitating the above discussion, we may also show that $\lvert \hat{y}\rvert$ is invertible in $LS(M)$. 
Thus $\lvert \hat{y}\rvert$ is an operator in $LS(M)$ that is invertible in $LS(M)$. 
Therefore, we may define a ring isomorphism $\Psi_1\colon LS(M)\to LS(N)$ by $\Psi_1(x):=v\lvert \hat{y}\rvert x \rvert \hat{y}\rvert^{-1} v^*$.
Then the corresponding lattice isomorphism $\Phi_1$ satisfies $\Phi_1(p)=l(v\lvert \hat{y}\lvert p)=l(\hat{y}p)$ for every $p\in \P(M)$. 
It follows that $\Phi(p)=l(\hat{y} p)=\Phi_1(p)$ for every $p\in \P(M)$ with $p\prec p_0$. 
Since every projection in $\P(M)$ can be written as the least upper bound of some family in $\{p\in \P(M)\mid p\prec p_0\}$, the two lattice isomorphisms $\Phi$ and $\Phi_1$ coincide on $\P(M)$.
Therefore, the corresponding ring isomorphisms $\Psi$ and $\Psi_1$ coincide, i.e., $\Psi(x)=v\lvert \hat{y}\rvert x\lvert \hat{y}\rvert^{-1} v^*$ for every $x\in LS(M)$. 
It follows that $\Psi$ is similar to a $^*$-isomorphism.
This completes the proof of Proposition \ref{special} and hence that of Theorem \ref{main}.

\section{Remarks}
Note that the operator $y$ in Theorems \ref{main}--\ref{projection} has a polar decomposition $y=u\lvert y\rvert=\lvert y^*\rvert u$ such that $u\in N$ is a unitary operator. 
Thus we may take the pair $(u\psi(\cdot)u^*, \lvert y^*\rvert)$ instead of $(\psi, y)$. 
This implies that one may take \emph{positive} invertible $y\in LS(N)$ in these theorems. 
Now let us discuss the uniqueness of $(\psi, y)$ with positive $y$. 

\begin{lemma}
Let $M, N$ be von Neumann algebras.
Let $\psi_1, \psi_2\colon M\to N$ be real $^*$-isomorphisms, and let $y_1, y_2$ be positive invertible operators in $LS(N)$. 
Then the equality $y_1\psi_1(x)y_1^{-1}=y_2\psi_2(x)y_2^{-1}$ holds for every $x\in LS(M)$ if and only if $\psi_1=\psi_2$ and there is a positive invertible operator $z\in LS(\Z(N))$ such that $y_2=zy_1$.  
\end{lemma}
\begin{proof}
The equality $y_1\psi_1(x)y_1^{-1}=y_2\psi_2(x)y_2^{-1}$ means 
\[
y_2^{-1}y_1\psi_1(x) (y_2^{-1}y_1)^{-1}= \psi_2(x) = \psi_2\circ \psi_1^{-1}(\psi_1(x)).
\]
Put $y:=y_2^{-1}y_1$, and $\psi:=\psi_2\circ \psi_1^{-1}\colon N\to N$. 
Then we have $\psi(a)=yay^{-1}$ for every $a\in LS(N)$.
Let $y=u\lvert y\rvert$ be the polar decomposition of $y$. Then $u$ is a unitary operator in $N$. 
We show that $\lvert y\rvert$ lies in the center $LS(\Z(N))$ of $LS(N)$.
If $\lvert y\rvert$ is not central, then the spectral theorem implies that $\chi_{[0, c)}(\lvert y\rvert)$ is not central for some $c>0$. 
We may take a nonzero partial isometry $v\in N$ such that $vv^*\leq \chi_{[0, c)}(\lvert y\rvert)$ and $v^*v\leq \chi_{[0, c)}(\lvert y\rvert)^\perp=\chi_{[c, \infty)}(\lvert y\rvert)$.
Then $\psi(v)=u\lvert y\rvert v\lvert y\rvert^{-1} u^*$ is not a partial isometry.
Indeed, we have $u^*\psi(v)u=\lvert y\rvert v\lvert y\rvert^{-1} = \lvert y\rvert \chi_{[0, c)}(\lvert y\rvert)v\chi_{[c, \infty)}(\lvert y\rvert)\lvert y\rvert^{-1}$, and we may check that this is a nonzero bounded operator with norm at most one such that $\lVert u^*\psi(v)uh\rVert< \lVert h\rVert$ for every nonzero vector $h$. 
This contradicts the assumption that $\psi_1, \psi_2$ are real $^*$-automorphisms. (Recall that a real $^*$-isomorphism is the direct sum of a $^*$-isomorphism and a conjugate-linear $^*$-isomorphism.)
Thus we see that $\lvert y\rvert$ is central.  
By the uniqueness of the polar decomposition of $y_1=y_2y=(y_2\lvert y\rvert)u$ together with the assumption that $y_1, y_2$ are positive, we obtain $u=1$ and $y=\lvert y\rvert \in LS(\Z(N))$. 
It follows that $\psi$ is the identity mapping, and we obtain the desired conclusion.
\end{proof}

We stated Theorem \ref{projection} in terms of real $^*$-isomorphisms, but we may equivalently use Jordan $^*$-isomorphisms.
Recall that a complex-linear bijection $J\colon A\to B$ between $^*$-algebras $A, B$ is called a \emph{Jordan $^*$-isomorphism} if it satisfies $J(x^*)=J(x)^*$ and $J(x^2)=J(x)^2$ for every $x\in A$. 
It is well-known that a Jordan $^*$-isomorphism between two von Neumann algebras can be decomposed into the direct sum of a $^*$-isomorphism and a $^*$-antiisomorphism. 
More precisely, if $J\colon M\to N$ is a Jordan $^*$-isomorphism between von Neumann algebras $M, N$, then there are central projections $r\in \P(\Z(M))$ and $z\in \P(\Z(N))$, a $^*$-isomorphism $\psi_1\colon rM\to zN$, and a $^*$-antiisomorphism $\psi_2\colon r^\perp M\to z^\perp N$ such that $J(x)=\psi_1(rx)+\psi_2(r^\perp x)$ for every $x\in M$, see e.g.\ \cite[Exercise 10.5.26]{KR}. 
For such a decomposition, notice that the map $\psi\colon M\ni x\mapsto \psi(x):=\psi_1(rx)+\psi_2(r^\perp x)^*\in N$ is a real $^*$-isomorphism. 
This gives a one-to-one correspondence between Jordan $^*$-isomorphisms and real $^*$-isomorphisms. 
Note that $J$ and $\psi$ coincide on the self-adjoint part of $M$, and in particular, on $\P(M)$. 
Therefore, we may replace the map $\psi$ in the statement of Theorem \ref{projection} with a Jordan $^*$-isomorphism.

\end{document}